\documentclass[12pt]{article}
\usepackage{amsmath,amssymb,graphicx,amsthm,float}

\newtheorem{theorem}{Theorem}[section]
\newtheorem*{theorem:repeat}{\tref{butterflystab}}
\newtheorem*{theorem:repeatmain}{\tref{main}}
\newtheorem{lemma}[theorem]{Lemma}
\newtheorem{corollary}[theorem]{Corollary}
\newtheorem{problem}[theorem]{Problem}
\newtheorem{prop}[theorem]{Proposition}

\newcommand\tref[1]{Theorem~\ref{thm:#1}}
\newcommand\cref[1]{Corollary~\ref{cor:#1}}

\newcommand{\w}{{\rm w}}

\newcommand\cF{{\mathcal F}}
\newcommand\cG{{\mathcal G}}
\newcommand\cH{{\mathcal H}}

\begin{document}
\title{\vspace{-2.9cm} ~~ \\ Domination game on uniform hypergraphs}

\author{
Csilla Bujt{\'a}s $^{a}$
\and
Bal{\'a}zs Patk{\'o}s $^{b}$
\and
Zsolt Tuza $^{a,b}$
\and
M{\'a}t{\'e} Vizer $^{b}$
}

\maketitle

\begin{center}
$^a$ Faculty of Information Technology, University of Pannonia, Veszpr\'em, Hungary\\
\medskip
$^b$ Alfr\'ed R\'enyi Institute of Mathematics, Hungarian Academy of Sciences, Budapest, Hungary\\
\medskip
\small \texttt{bujtas,tuza@dcs.uni-pannon.hu, patkos@renyi.hu, vizermate@gmail.com}
\medskip
\end{center}

\begin{abstract}
In this paper we introduce and study the domination game on hypergraphs. This is played on a hypergraph $\cH$ by two players, namely Dominator and Staller, who alternately select vertices such that each selected vertex enlarges the set of vertices dominated so far. The game is over if all vertices of $\cH$ are dominated. Dominator aims to finish the game as soon as possible, while Staller aims to delay the end of the game. If each player plays optimally and Dominator starts, the length of the game is the invariant `game domination number' denoted by $\gamma_g(\cH)$. This definition is the generalization of the domination game played on graphs and it is a special case of the transversal game on hypergraphs. After some basic general results, we establish an asymptotically tight upper bound on the game domination number of $k$-uniform hypergraphs. In the remaining part of the paper we prove that $\gamma_g(\cH) \le 5n/9$ if $\cH$ is a 3-uniform hypergraph of order $n$ and does not contain isolated vertices. This also implies the following new result for graphs: If $G$ is an isolate-free graph on $n$ vertices and each of its edges is contained in a triangle, then $\gamma_g(G) \le 5n/9$.

\end{abstract}


\noindent
{\bf AMS Subj.\ Class.\ (2010)}: 05C69, 05C76

\section{Introduction}

For a graph $G = (V, E)$ and for a vertex $v \in V$, the \textit{open
neighborhood} $N(v)$ of $v$ is the set of all vertices adjacent to $v$, and its \textit{closed
neighborhood} is $N[v] := N(v) \cup \{v\}$. Each vertex \textit{dominates} itself and its neighbors, moreover a set $S \subseteq V$ dominates the vertices contained in $N[S] :=
\bigcup_{v \in S} N[v]$. A \textit{dominating set} of $G$ is a subset $D$ of $V$ which dominates all vertices of the graph; that is $N[D] = V$. The minimum cardinality of a
dominating set, denoted by $\gamma(G)$, is termed the \textit{domination number} of $G$.

\vspace{2mm}

Domination is a well-studied subject in graph theory, with many related
applications. A general overview can be found in \cite{HHS1998}.

\vspace{2mm}

Recently Bre\v{s}ar, Klav\v{z}ar and Rall~\cite{BKR-SIAM} introduced the so-called \textit{domination game}. It is played on a graph $G=(V,E)$ by two players, \textit{Dominator} and \textit{Staller}. They take turns choosing a vertex from $V$ such that at least one new vertex must be dominated in each turn, that we call a \textit{legal move}. The game ends when no more legal moves can be taken. Note that the set of chosen vertices forms a dominating set. In this game Dominator's aim is to finish as soon as possible (so to finish the game with a small dominating set), while Staller's aim is to delay the end of the game. 
The \textit{game domination number} is the number of turns in the
game when the first turn is Dominator's move and both players play optimally. 
For recent results on this topic see \cite{BDKK-2017,BKR,B1,B2,DKR-2015,HK,KWZ,K-2014,S}.

\vspace{2mm}

The domination game belongs to the growing family  of
\textit{competitive  optimization} graph  and  hypergraph  games.   
Competitive  optimization variants of coloring \cite{KK2009,TZ2014}, list-colouring \cite{BST2007}, transversal \cite{BHT2}, matching \cite{CKOW2013}, Ramsey theory \cite{GHK2004},
and others \cite{PV2014} have been extensively investigated.

\subsection*{Domination in hypergraphs}

Despite that domination is a well-investigated notion in graph theory, domination in hypergraphs is a relatively new subject. It was introduced in~\cite{A2007}; for more recent results and references see~\cite{BHT, BPTV2017, HL-2012, HL-2016,KLDS}. 

\vspace{2mm}

We say that in a hypergraph $\cH=(V(\cH), E(\cH))$, two different vertices $u,v \in V(\cH)$ are \emph{adjacent} if there is an edge $e \in E(\cH)$ with $\{u,v\}\subseteq e$. The \emph{open neighborhood} of a vertex $v \in V(\cH)$, denoted by $N_{\cH}(v)$ (or
simply $N(v)$ if $\cH$ is clear from the context), is the set of all
vertices that are adjacent to $v$. The \emph{closed neighborhood} of $v \in V(\cH)$ is the set $N[v]:=N(v)\cup \{v\}$, and if $S \subseteq V(\cH)$, its closed neighborhood is $N[S]= \bigcup_{v \in S} N[v]$. We say that $S \subset V(\cH)$ is a \textit{dominating set}, if $N[S] = V(\cH)$; and the domination number $\gamma(\cH)$ of $\cH$ is the minimum cardinality of a dominating set.

\vspace{2mm}

\noindent
\textbf{Domination game on hypergraphs.}
We define the \textit{domination game on hypergraphs} analogously to that on graphs. Given a hypergraph $\cH$, in the domination game  two players, namely Dominator and Staller, take turns choosing a vertex of $\cH$. We denote by $v_1,\dots, v_d$ the sequence of vertices selected by the players in the game and denote by $D_i$ the set $\{v_j: 1 \le j \le i\}$.  The sequence $v_1,\dots, v_d$ defines a legal dominating game if and only if  $N[v_{i}]\setminus N[D_{i-1}]\neq \emptyset$ for all $2\le i\le d$ and $D_d$ is a dominating set of $\cH$. Dominator's goal is to finish the game as soon as possible, while Staller's goal is to delay the end. When both players play optimally and Dominator (resp.\ Staller) starts the game, the uniquelly defined length of the game is the \textit{game domination number} $\gamma_g(\cH)$ (resp.\ the \textit{Staller-start game domination number} $\gamma_g'(\cH)$).

\subsection*{Structure of the paper}

In Section 2 we introduce related notions, state connection of the domination game on hypergraphs with the transversal game on hypergraphs and with domination game on certain graphs. In Section 3 we state our results. In Section 4 we prove our theorem on the asymptotics of the game domination number of $k$-uniform hypergraphs.
In Section 5 we prove an upper bound on the game domination number of 3-uniform hypergraphs. Finally in Section 6 we pose some open problems.

\vskip 0.3truecm

We will use the following standard notation: for any positive integer $n$, we denote by $[n]$ the set $\{1,2,\dots,n\}$ of the first $n$ positive integers, and for any set $S$ we use $\binom{S}{k}$ to denote the set $\{T: T\subset S, |T|=k\}$ of all $k$-element subsets of $S$.

\section{Preliminaries}

For a hypergraph $\cH=(V(\cH),E(\cH))$, we associate the \textit{closed neighborhood hypergraph} $CNH(\cH)$ to $\cH$, defined on the same set $V(\cH)$ of vertices, by setting the edges as $$E(CNH(\cH)):= \{ N_\cH[x] : x \in V(\cH)\}.$$  Let us note that these hypergraphs are in fact multi-hypergraphs, since different vertices may have the same neighborhood.

\vspace{2mm}

The 2-section graph $[\cH]_2$ of a hypergraph $\cH=(V(\cH),E(\cH))$ is the graph whose vertex set is $V(\cH)$, and  any two vertices $u$ and $v$ are adjacent in $[\cH]_2$ if and only if   they are adjacent in $\cH$. By definition, $N_{\cH}[v]=N_{[\cH]_2}[v]$ for every vertex $v$, and we have $CNH(\cH)=CNH([\cH]_2)$. Note that $[\cH]_2=\cH$ if $\cH$ is a 2-uniform hypergraph that is a simple graph.

\vspace{4mm}

\subsection*{Transversal number, game transversal number}

\vspace{2mm}

A subset $T$ of vertices in a hypergraph $\cH$ is a \textit{transversal} (also called \textit{hitting set} or \textit{vertex cover} in many papers) if $T$ has a nonempty intersection with every edge of $\cH$. A vertex hits or covers an edge if it belongs to that edge. The \textit{transversal number} $\tau(\cH)$ of $\cH$ is the minimum size of a transversal in $\cH$. In hypergraph theory the concept
of transversal is fundamental and well studied. The major monograph [1] of hypergraph
theory gives a detailed introduction to this topic, for recent results see \cite{BHT,BHT3} and the references therein.

\vspace{2mm}

The \textit{transversal game} played on a hypergraph $\cH$ involves two players, \textit{Edge-hitter} and \textit{Staller}, who take turns choosing a vertex from $\cH$. Each vertex chosen must hit at least
one edge not hit by the vertices previously chosen. We call such a chosen vertex a legal
move in the transversal game. The game ends when the set of vertices chosen becomes a
transversal in $\cH$. Edge-hitter wishes to end the game with the smallest possible number of 
vertices chosen, and Staller wishes to end the game with as many vertices chosen as possible.
The \textit{game transversal number} (resp.\ \textit{Staller-start game transversal number}), $\tau_g(\cH)$ (resp.\
$\tau'_g(\cH)$), of H is the number of vertices chosen when Edge-hitter (resp.\ Staller) starts the
game and both players play optimally.

\vspace{2mm}

Given a hypergraph $\cH$ and a subset $S \subseteq V(\cH)$, we denote by $\cH|S$ the residual hypergraph, in which the vertices contained in $S$ are declared to have been already dominated. That is, $D$ is a dominating set of $\cH|S$ if $N_\cH[D] \cup S=V(\cH)$. We write $\gamma_g(\cH|S)$
(resp.\ $\gamma'_g(\cH|S)$) for the number of further steps in the game
under optimal play when Dominator (resp.\ Staller) has the next turn.
Two games on $\cH_1$ and $\cH_2$ are called \textit{equivalent} if $V(\cH_1)=V(\cH_2)$ and any sequence  $v_1,\dots, v_d$ is a legal game on $\cH_1$ if and only if it is legal on $\cH_2$.

\begin{prop} \label{equiv1}
For every hypergraph $\cH$ and $S \subseteq V(\cH)$ we have the following:
\begin{itemize}
\item[$(i)$] the domination game on $\cH$ is equivalent to the domination game on $[\cH]_2$ and to the transversal game on $CNH(\cH)$;
\item[$(ii)$] $\gamma_g(\cH)=\gamma_g([\cH]_2)=\tau_g(CNH(\cH))$; and
\item[$(iii)$] $\gamma_g(\cH|S)=\gamma_g([\cH]_2|S)$.
\end{itemize}
\end{prop}

\begin{proof} For every vertex $v \in V(\cH)$, its closed neighborhood is the same in $\cH$ as in $[\cH]_2$. By definitions, this fact directly implies $(i)-(iii)$.
\end{proof}

By Proposition \ref{equiv1}, several results which were proved for the domination game on graphs can be generalized for the game on hypergraphs. First we state the so-called ``Continuation Principle'' for hypergraphs, which follows from the corresponding results in \cite{KWZ} and also from that in  \cite{BHT3}. Proposition~\ref{5/8} below is the consequence of Proposition \ref{equiv1} and the inequality $\gamma_g(G)\le \frac{5}{8}n$ which holds for every isolate-free graph $G$ \cite{B2017+}.

\begin{corollary}[Continuation Principle for Hypergraphs]
Given a hypergraph $H$ and vertex sets $A$ and $B$ such that $A\subseteq B\subseteq  V(\cH)$, we have $$\gamma_g(\cH|A) \ge \gamma_g(\cH|B).$$
\end{corollary}
 
\begin{prop}\label{5/8}
If $\cH$ is a hypergraph on $n$ vertices, and $\cH$ does not contain any isolated vertices and 1-element edges, then we have
$$\gamma_g(\cH)\le \frac{5}{8}n.$$
\end{prop}
Due to the `3/5-conjecture' \cite{KWZ} the game domination number of an isolate-free graph $G$ of order $n$ is expected to be at most $3n/5$. Once this conjecture will be proved, also the upper bound $5n/8$ can be replaced with $3n/5$ in Proposition~\ref{5/8}.

\section{Our results}

We will be interested in the maximum possible value of the ratio $\frac{\gamma_g(\cH)}{|V(\cH)|}$ over all $k$-uniform hypergraphs.

\subsection*{Results on $k$-uniform hypergraphs}

Our first theorem establishes a bound on $\frac{\gamma_g(\cH)}{|V(\cH)|}$ for all $k$-uniform hypergraphs. We also show that it is asymptotically sharp as $k$ tends to infinity.

\begin{theorem}\label{gdkunif}
\textbf{(a)} If $\cH_k$ is a sequence of $k$-uniform hypergraphs, then $$\frac{\gamma_g(\cH_k)}{|V(\cH_k)|}\le (2+o_k(1))\frac{\log k}{k}.$$

\textbf{(b)} There exists a sequence $\cH_k$ of $k$-uniform hypergraphs with $$\frac{\gamma_g(\cH_k)}{|V(\cH_k)|}\ge (2-o_k(1))\frac{\log k}{k}.$$
\end{theorem}

\medskip

\noindent
A vertex set $D \subseteq V(\cH)$ is a \textit{total dominating set} in a hypergraph $\cH$ if $\bigcup_{v\in D} N_\cH(v)=V(\cH)$. The minimum cardinality of a total dominating set is the \textit{total domination number} $\gamma_t(\cH)$ of the hypergraph, introduced in \cite{BHT2014}. If one uses total domination instead of domination (that is open neighborhoods instead of closed ones) in the definition of the game domination number, then it is called the \textit{game total domination number} of a hypergraph, and we denote it by $\gamma_{tg}(\cH)$. Total domination is a well-studied  notion in graphs \cite{HY2013}, the game total domination number is also a recently introduced \cite{HKR2015} active area \cite{B2017, DH-2016, HKR2017}. We could achieve the same result as Theorem \ref{gdkunif} for the game total domination number of $k$-uniform hypergraphs. We state the result, but omit its proof as it is analogous to the proof of Theorem \ref{gdkunif}, which will be given in Section \ref{s:pfkunif}.

\begin{theorem}
\textbf{(a)} If $\cH_k$ is a sequence of $k$-uniform hypergraphs, then $$\frac{\gamma_{tg}(\cH_k)}{|V(\cH_k)|}\le (2+o_k(1))\frac{\log k}{k}.$$

\textbf{(b)} There exists a sequence $\cH_k$ of $k$-uniform hypergraphs with $$\frac{\gamma_{tg}(\cH_k)}{|V(\cH_k)|}\ge (2-o_k(1))\frac{\log k}{k}.$$
\end{theorem}

\subsection*{Results on 3-uniform hypergraphs}

After determining the asymptotics of the maximum of $\gamma_g$ among $k$-uniform hypergraphs, we concentrate on the domination game on 3-uniform hypergraphs. Denoting by $\cG_{2,3}$  the class of graphs in which each edge ($K_2$) is contained in a triangle ($K_3$), we can observe the following correspondence.

\begin{prop} \label{equiv2}
For every 3-uniform hypergraph $\cH$ there exists a graph $G \in \cG_{2,3}$ and vice versa, for every graph $G$ from $\cG_{2,3}$ there exists a 3-uniform hypergraph $\cH$, such that the domination game on $\cH$ is equivalent to that on $G$.
\end{prop}
\proof Given a 3-uniform hypergraph $\cH$, consider its 2-section graph $[\cH]_2$ and observe that every edge of  $[\cH]_2$ belongs to a triangle. Hence, $[\cH]_2 \in \cG_{2,3}$. On the other hand, given a graph  $G \in \cG_{2,3}$ we can construct a 3-uniform hypergraph $\cH$ on the same vertex set such that three vertices form a hyperedge in $\cH$ if and only if they induce a triangle in $G$. Then, we have $G=[\cH]_2$.  In both cases, by Proposition~\ref{equiv1} $(i)$, the domination game played on $\cH$ is equivalent to that played on $G$. \qed

\bigskip

For any 3-uniform hypergraph $\cH$ on $n$ vertices, $\gamma_g(\cH) \le 3n/5$ follows from Proposition~\ref{equiv2}  and from the recent result of \cite{HK} where Kinnersley and Henning proved the 3/5-conjecture for graphs of minimum degree at least 2.
In Section \ref{pf:main} we will show that the bound on $\gamma_g(\cH)$ can be improved to $5n/9$.

\begin{theorem} \label{5/9}
\begin{itemize}
\item[$(i)$]
If $\cH$ is an isolate-free 3-uniform hypergraph on $n$ vertices, then
$$\gamma_g(\cH)\le \frac{5}{9}n.$$
\item[$(ii)$] 
If $G$ is an isolate-free graph of order
$n$ and each edge of $G$ belongs to a triangle, then
$$\gamma_g(G)\le \frac{5}{9}n.$$
\end{itemize}
\end{theorem}
\medskip
By Proposition~\ref{equiv2}, the two statements are clearly equivalent. In the proof we will consider a graph $G\in \cG_{2,3}$ and assign weights to the vertices which will be changed during the game reflecting on the current status of the vertices. 

\vspace{2mm}

In the case of 3-uniform hypergraphs the best lower bound we can prove is the following:

\begin{prop} \label{3unilower}

There exists a 3-uniform hypergraph $\cH$ with

$$\gamma_g(\cH)=\gamma'_g(\cH)=\frac{4}{9}|V(\cH)|.$$

\end{prop}

\begin{proof}


Let $\cH$ be the hypergraph consisting of all but one lines of the $3\times 3$ grid. Formally, let $V(\cH)=\{1,2\dots,9\}$ and $E(\cH)=\{\{1,2,3\},\{4,5,6\},\{1,4,7\},\{2,5,8\}, \{3,6,9\}\}$. We claim that $\gamma_g'(\cH)=\gamma_g(\cH)=4$. For the ordinary game Staller in his first move must make sure that after his move there still exist 2 non-adjacent vertices that are undominated. Because of symmetry, we may assume that the first move of Dominator is either 1 or 9 (only the degree of the vertex picked matters). In both cases, Staller can pick 4 (it dominates 5 and 6), and 8 and 9 will be the undominated non-adjacent vertices. Dominator may insist on playing vertices of degree two whose no neighbors have been played so far in the game. The Staller start game can be analyzed similarly.
\end{proof}

\section{Proof of Theorem \ref{gdkunif}}
\label{s:pfkunif}

\begin{proof}
To prove \textbf{(a)} we consider an arbitrary $k$-uniform hypergraph $\cH$ on $n$ vertices and its closed neighborhood hypergraph $\cF=CNH(\cH)$. Then, every edge of $\cF$ contains at least $k$ vertices and, by Proposition~\ref{equiv1} $(ii)$ we have $$\gamma_g(\cH)=\tau_g(\cF).$$ Since  $$\tau_g(\cF) \le 2\tau(\cF)-1$$ (a simple proposition from \cite{BHT2}), and $$\tau(\cF)\le \frac{1+\log k}{k}\;n$$ (proved in \cite{Alon}), we have $$\frac{\gamma_g(\cH)}{n}<2\cdot\frac{1+\log k}{k},$$
and we are done with the proof of \textbf{(a)}.

The construction showing \textbf{(b)} is based on Alon's construction from \cite{Alon}. He showed that if the vertex set of $\cH'_k$ is $[\lceil k\log k\rceil]$ and the edge set of $\cH'_k$ consists of $k$ many $k$-subsets of $V(\cH'_k)$ chosen uniformly at random, then the transversal number of $\cH'_k$ is at least $(1-o(1))\log^2k$ with probability tending to 1 as $k$ tends to infinity. Our hypergraph $\cH_k$ is an extension of $\cH'_{k-1}$. Its vertex set is $$V(\cH_k):=[\lceil (k-1)\log (k-1)\rceil+k-1]$$ and its edge set is $$E(\cH_k):=\{e_i \cup[\lceil (k-1)\log (k-1)\rceil+i]:1\le i\le k-1 \},$$ where $e_1,e_2,\dots, e_{k-1}$ are chosen uniformly at random from $$\binom{[\lceil (k-1)\log (k-1)\rceil]}{k-1}.$$

Observe that if $D$ is a dominating set of $\cH_k$, then it is also a transversal of $\cH_k$ as $D$ must contain an element of $e_i$ to dominate $\lceil (k-1)\log (k-1)\rceil+i$.
Therefore Staller's strategy is the following: as long as there is an edge $e_i$ that is disjoint from the set of already chosen vertices, she picks the vertex $\lceil (k-1)\log (k-1)\rceil+i$. 

How long can the game last? As only Dominator selects vertices from $V(\cH'_{k-1})$, after $m$ rounds there are at most $m$ such vertices and at most $2m$ vertices from $V(\cH_k)\setminus V(\cH'_{k-1})$. Therefore if for any $m$-subset $X$ of $V(\cH'_{k-1})$ the number of edges of $\cH'_{k-1}$ that are disjoint from $X$ is at least $2m$, then the game lasts for at least $2m$ steps. As $$\frac{\log (k-1)}{k-1}=(1+o(1))\frac{\log k}{k},$$ the following lemma will finish the proof of \textbf{(b)}.

\begin{lemma}
For any set $X$ of size $\lfloor\log^2k-10\log k\log\log k\rfloor$ there exist at least $\frac{\log^9k}{2}$ edges of $\cH'_k$ that are disjoint from $X$, with probability tending to 1 as $k\to\infty$.
\end{lemma}

\begin{proof}[Proof of Lemma]
Let us fix subset $X$ of $V(\cH'_k)$ of size $\lfloor\log^2k-10\log k\log\log k\rfloor$. Alon \cite{Alon} calculated that for a $k$-set $e$ chosen uniformly at random from $\binom{V(\cH'_k)}{k}$ we have $$\mathbb{P}(e\cap X=\emptyset)\ge \frac{\log^9 k}{k}.$$
Therefore if we introduce the random variable $Z_X$ as the number of edges in $\cH'_k$ that are disjoint from $X$, we obtain that $$\mathbb{E}(Z_X)\ge \log^9 k.$$ Applying Chernoff's inequality, we obtain $$\mathbb{P}\left(Z_X<\frac{\log^9k}{2}\right)\le e^{-c\log^9k}$$ for some constant $c > 0$. As the number of subsets of size $\lfloor\log^2k-10\log k\log\log k\rfloor$ is $$\binom{\lceil k\log k\rceil}{\lfloor\log^2k-10\log k\log\log k\rfloor}< (k\log k)^{\log^2k-10\log k\log\log k}< e^{\log^4k}$$ we obtain
\[
\mathbb{P}\left(\exists X:Z_X<\frac{\log^9k}{2}\right)\le e^{log^4k-c\log^9k}=o(1). 
\]
\end{proof}
\noindent
We are done with the proof of Theorem \ref{gdkunif}.
\end{proof}

\section{Proof of Theorem \ref{5/9}}
\label{pf:main}


\subsection*{Colors, residual graphs and weights}

We consider a domination game played on a graph $G^*$ and assume that each vertex and edge of $G^*$ is contained in at least one triangle. The vertex which is played in the $i$th turn of the game will be denoted by $p_i$; we also use the notation $D_i=\{p_1,\dots, p_i\}$ for $i\ge 1$, and set $D_0=\emptyset$. By definition, $N[p_i]\setminus N[D_{i-1}]\neq \emptyset$ holds for each $i$. Note that $p_i$ is played by Dominator if $i$ is odd, otherwise it is played by Staller. We assign colors to the vertices during the game which reflect the current dominated/nondominated status of the vertices and their neighbors. After the $i$th turn of the game, for a vertex $v\in V(G^*)$:
\begin{itemize}
\item $v$ is \emph{white}, if $v\notin N[D_i]$;
\item $v$ is \emph{blue}, if $v\in N[D_i]$ but $N[v] \nsubseteqq N[D_i]$;
\item $v$ is \emph{red}, if $N[v] \subseteqq N[D_i]$.
\end{itemize}
The set of white, blue and red vertices are denoted by $W$, $B$ and $R$, respectively. That is, only the vertices from $B\cup R$ are dominated, and only the vertices from $W\cup B$ might be played in the continuation of the game. Before the first turn, every vertex of $G^*$ is white; in a turn when a vertex $v$ becomes dominated, its color turns from white to blue or red; if all neighbors of a blue vertex $u$ become dominated, $u$ is recolored red. Note that no other types of color-changes are possible. The game ends when $W=B=\emptyset$. If a (white or blue) vertex is played, it will be recolored red, all its white neighbors turn blue or red, and there might be some further blue vertices which are recolored red.

It was observed in several earlier papers (e.g., \cite{KWZ, B1, S}) that the red vertices and the edges connecting blue vertices do not effect the continuation of the game, they can be deleted. After the $i$th turn, deleting these vertices and edges, we obtain the \emph{residual graph} $G_i$  whose vertices are considered together with the assigned white or blue colors. Observe that if a vertex is white in $G_i$, then none of its neighbors in $G^*$ and none of the incident edges were deleted. After the last turn of the game the residual graph is empty (without vertices); we also define $G_0:=G^*$ with  all vertices colored white which is the residual graph at the beginning of the game.

In a residual graph $G_i$, the \emph{W-degree} $\deg_{G_i}^W(v)$ (or simply $\deg^W(v)$) of a vertex $v$ is the number of white neighbors of $v$. Similarly, the \emph{WB-degree} $\deg_{G_i}^{WB}(v)$ equals the number of (white and blue) neighbors in $G_i$. We also introduce the following notations and concept: $W_j$ denotes the set of white vertices of  W-degree $j$; $B_j$ denotes the set of blue vertices of  W-degree $j$; we say that a $K_3$ subgraph of $G_i$ is a \emph{special triangle}, if it contains one blue and two white vertices; $B_T$ denotes the set of those blue vertices which are contained in at least one special triangle.

Given a residual graph $G_i$, we assign a non-negative weight $\w(v)$ to every vertex $v\in V(G_i)$, and the weight of $G_i$ is defined as $\w(G_i)=\sum_{v\in V(G_i)}\w(v)$. The weight assignemnt will be specified such that $\w(G_i) \le \w(G_{i-1})$ holds for any two consecutive residual graphs, and the decrease $\w(G_{i-1})-\w(G_{i})$ will be denoted by $g_i$. We will have $\w(G_0)=20 n$ at the beginning of the game and the weight will be zero at the end. Thus, it is enough to prove that Dominator has a strategy which makes sure that the average decrease of $\w(G_i)$ in a turn is at least $36$. Throughout, we suppose that Dominator follows a greedy strategy; that is, in each of his turns, he chooses a vertex such that the decrease $g_i$ is maximum possible.

\subsection*{Phase 1}

First, we define Phase 1 and 2 of the domination game. If $i$ is odd and $G_{i-1}$ contains a white vertex $v$ with $\deg^W(v)\ge 4$ or a blue vertex $u$ with $\deg^W(u)\ge 5$, then the $i$th and the $(i+1)$st (if exists) turns belong to Phase 1.
Otherwise, these two turns belong to Phase 2. Observe that once the game is in Phase 2, it will remain in this phase until the end. We also remark that Phase~1 always ends after one of Staller's moves with the only exception when Dominator finishes the game by playing a vertex $v$ with $v \in W$ and $\deg^W(v)\ge 4$, or with $v \in B$ and $\deg^W(v)\ge 5$. 

In Phase 1, the weights of the vertices are determined as shown in Table~1.
\begin{center}
\begin{tabular}{|c|c|}  \hline
\emph{Color of vertex} &  \emph{Weight of vertex} \\
 \hline
white &  $20$ \\
blue & $12$  \\
red & $0$  \\
 \hline
\end{tabular}
\end{center}
\begin{center}
\textbf{Table~1.} The weights of vertices in Phase 1 according to their color.
\end{center}

\begin{lemma} \label{lem-Ph1}
If $i$ is odd and the $i$th turn belongs to Phase~1, then either $g_i+g_{i+1} \ge 72$ or the game ends with the $i$th turn and $g_i > 36$.
\end{lemma}
\proof Consider the residual graph $G_{i-1}$, and assume first that $v$ is a white vertex with $\deg^W(v)\ge 4$. If Dominator plays $v$, it will be recolored red and each white neighbor of $v$ turns either blue or red. Hence, the decrease $g_i$ in the weight of the residual graph is at least $20+4(20-12)=52$. Note that Dominator might play another vertex instead of $v$ but, by his greedy strategy, the decrease cannot be smaller than 52. In the other case, if no such white $v$ exists, there is a blue vertex $u$ with $\deg^W(u)\ge 5$. If Dominator selects $u$, it turns red and all the at least 5 white vertices in $N_{G_{i-1}}(u)$ are recolored blue or red. Hence, in this case $g_i \ge 12+5(20-12)=52$ holds again. Now, consider Staller's turn. By the rule of the game, he must dominate at least one new vertex. Then, the played vertex $p_{i+1}$ is either white and turns red, which results in $g_{i+1}\ge 20$, or  $p_{i+1}$ is blue and has a white neighbor $v$. In this latter case, $p_{i+1}$ is recolored red and $v$ is recolored blue or red that yields $g_{i+1}\ge 12+(20-12)=20$. Hence, $g_i+g_{i+1} \ge 72$ and the second part of the statement also follows as $g_i\ge 52$. \qed 

\subsection*{Phase 2}

Assume that the game did not finish in the first phase, and denote by $i^*$ the length of Phase~1. Hence, $i^*$ is even, and Phase~2 begins with the residual graph $G_{i^*}$ in which every white vertex has W-degree of at most 3, and every blue vertex has W-degree of at most 4. First, we change the weight assignment. The weight of a blue vertex $v$ in $G_i$ ($i\ge i^*$) will depend on the following three factors: its W-degree, containment in special triangles, and on the color of $v$ in $G_{i^*}$. From now on, the weight of $v$ in a residual graph $G_i$  is defined as  $\w(v)=f(v)+f^+(v)$, where $f(v)$ is determined according to Table~2, and $ f^+(v)$ is given here:
$$ f^+(v)= \left\{
 \begin{array}{ll}
  2 \enskip & \mbox{if \enskip  $\deg_{G_i}^W(v) \ge 2$ \enskip and \enskip $v$ was blue in $G_{i^*}$;}\\ \\
  
  0 & \mbox{otherwise.}
  \end{array}
 \right.
 $$
\medskip

\begin{center}
\begin{tabular}{|c|c|}  \hline
\emph{Color/Type of vertex $v$ in $G_i$}  & \emph{\quad $f(v)$ \quad} \\
 \hline
$v$ is white &  $20$ \\
$v \in B_4$  & $10$  \\
$v \in (B_3 \cup B_2)\cap B_T$ & $10$  \\
$v \in (B_3 \cup B_2)\setminus B_T$ & $9$  \\
$v \in B_1$  & $8$  \\
$v$ is red  & $0$  \\ \hline
\end{tabular}
\end{center}
\begin{center}
\textbf{Table~2.} Definition of $f(v)$.
\end{center}

By this definition, when the assignment is changed, none of the vertices may have larger weight at the beginning of Phase~2 than it had before at the end of Phase~1.
Further, for any fixed vertex $v$, neither $f(v)$ nor $f^+(v)$ may increase during Phase~2. We remark that the decrease in $f^+(v)$ will be referred to in only one special part of the proof of the following lemma.
\begin{lemma} \label{lem-Ph2}
If $i$ is odd and the $i$th turn belongs to Phase~2, then either $g_i+g_{i+1} \ge 72$ or the game ends with the $i$th turn and $g_i \ge 36$.
\end{lemma}
\proof
We will consider nine cases concerning the structure of the residual graph $G_{i-1}$. They together cover all possibilities which we may have in Phase~2.

\paragraph{(i)} There exist three pairwise adjacent white vertices, say $u$, $v$ and $w$.\\
If one of the three vertices, say $u$, belongs to $W_3$, let us choose $p_i=u$.
Then, $u$ is recolored red; $v$ turns red or blue, but in the latter case it will have only one remaining white neighbor and hence, $f(v)$ decreases by at least 12; the case is similar for $w$; the third neighbor of $u$ is recolored blue or red and hence, its weight decreases by at least 10. Consequently, $g_i \ge 20+ 2\cdot 12 +10=54$.
In the other case, none of $u$, $v$ and $w$ belongs to $W_3$ in $G_{i-1}$. Then, after the choice $p_i=u$, all the three vertices are recolored red and $g_i \ge 60$ follows.
Staller either plays a white vertex which is recolored red and $g_{i+1}\ge 20$, or he plays a blue vertex $v$ with a white neighbor $u$. In this case, $v$ turns red and $f(v)$ decreases by at least 8, while $u$ turns blue or red which results in a decrease of at least 10. Hence, $g_{i+1}\ge 18$ and in any case $g_i + g_{i+1}\ge 54+18=72$ follows.
\medskip

For the remaining cases, we observe that if (i) is not true for $G_{i-1}$, there is no  triangle induced by three white vertices. Then, if a white vertex $v$ is recolored blue,  $v$ cannot be contained in any special triangles and $f(v)$ decreases by at least 11.

\paragraph{(ii)} $W_3 \cup B_4 \neq \emptyset$, but (i) does not hold.\\
If there exists a $v\in W_3$ and Dominator plays it, $f(v)$ decreases by 20. Further, each  white neighbor of $v$ is recolored either blue or red and its weight decreases by at least 11. Therefore, $g_i \ge 20 + 3 \cdot 11 =53$. Similarly, if there exists a vertex $u \in B_4$ and it is played, $u$ is recolored red, and each white neighbor turns blue or red. This yields $g_i \ge 10 + 4 \cdot 11 =54$.
Concerning Staller's turn, if he plays a white vertex, $g_{i+1}\ge 20$ immediately follows. If he plays a blue vertex $v$ with a white neighbor $u$, $f(v)$ decreases by at least 8 and $f(u)$ decreases by at least 11. Hence, $g_{i+1}\ge 19$ and $g_i + g_{i+1}\ge 53+19=72$ that proves the lemma for the case (ii).
\medskip

From now on, we will assume that (i) and (ii) are not valid. Hence, in $G_{i-1}$, every white vertex is of W-degree 0, 1 or 2, and every blue vertex is of W-degree 1, 2, or 3.

\paragraph{(iii)} There exists an edge $uv$ with $u\in W_1$ and $v\in W_2$  in $G_{i-1}$, but (i) and (ii) are not valid.\\
Assume that Dominator plays $v$. Then $v$ is recolored red and $f(v)$ decreases by 20. Since $u$ had only one white neighbor, namely $v$, and in this turn $u$ and $v$ become  dominated, $u$ is also recolored red and $f(u)$ decreases by 20. Consider the other white neighbor $w$ of $v$. Since $\deg_{G_{i-1}}^W(w) \le 2$, in $G_i$ either $w$ is red or it is a blue vertex with $\deg_{G_{i}}^W(w)= 1$. Hence, $f(w)$ decreases by at least 12. By the condition of our theorem, the edge $vw$ was contained in a triangle in the graph $G^*=G_0$. Let us denote by $w'$ its third vertex. In $G_{i-1}$, $w'$ cannot be red (because it has white neighbors) and cannot be white because our present condition excludes the case (i). So, $w'$ belongs to $(B_2 \cup B_3)\cap B_T$ in $G_{i-1}$. When  $v$ and $w$ become dominated, $w'$ is either recolored red or recolored blue with $\deg_{G_{i}}^W(w')= 1$. Therefore, $f(w')$ decreases by at least 2.
The case is similar for the edge $vu$. For the blue vertex $u'$, which is adjacent to both $v$ and $u$, $f(u')$ decreases by at least 2, and if $w'=u'$, it becomes red and the decrease is higher. Thus, $g_i \ge 20 + 20+ 12+2\cdot 2 =56$. Concerning Staller's move, the same argumentation as given in (ii) proves $g_{i+1}\ge 19$. We may conclude that $g_i + g_{i+1}\ge 56+19=75$  holds.

\medskip

Observe that if none of (i)--(iii) is valid, then each component of the subgraph $G_{i-1} [W]$ induced by the white vertices in $G_{i-1}$ is either a cycle of length at least $4$, or $K_2$ or $K_1$. 

\paragraph{(iv)} There exist three pairwise adjacent vertices $u$, $v$, $w$ in $G_{i-1}$ such that $u,v \in W_1$ and $w\in B_3$, but none of (i)--(iii) are valid.\\
If Dominator plays $w$ (that is a blue vertex in a special triangle), then $u$, $v$ and $w$ are recolored red. Moreover, the third white neighbor $x$ of $w$ turns blue. Hence, 
$g_i \ge 10 + 2 \cdot 20+ 11 =61$, Since $g_{i+1}\ge 19$, we have $g_i + g_{i+1}\ge 61+19=80$.

\paragraph{(v)} There exists a blue vertex $v$ which belongs to two different special triangles, and none of (i)--(iv) are valid.\\
Since (i) and (ii) are not valid and $v$ is contained in two special triangles, $N(v)$ induces a path $u_1u_2u_3$ with $u_1,u_2,u_3$ being white vertices. If Dominator plays $v$, then $u_2$ and $v$ are recolored red, while $u_1$ and $u_3$  become blue vertices with a W-degree of 1. Then, $g_i \ge 10 + 20 + 2\cdot 12 =54$ follows. In the next turn $g_{i+1}\ge 19$ holds again and we conclude $g_i + g_{i+1}\ge 54+19=73$.
\medskip

At this point we prove separately that if none of (i)--(v) hold, then the following two claims are true.
\paragraph{Claim (a)} If $u$ and $v$ are two adjacent white vertices in $G_{i-1}$, then there exists a blue vertex $w$ such that $f(w)$ decreases if at least one of $u$ and $v$ becomes dominated. In particular, if $u,v\in W_1$, $f(w)$ decreases by at least 2.\\
Indeed, by the assumption of our theorem, the edge $uv$ was contained in a triangle $uvw$ in $G^*$. In $G_{i-1}$, the vertex $w$ cannot be red since $u$ and $v$ are undominated; and $w$ cannot be white as case (i) is excluded. Hence, $uvw$ is a special triangle, $w\in B_T$, and $f(w)= 10$ in $G_{i-1}$. Since case (v) is also excluded, $uvw$ is the only special triangle incident to $w$. When  $u$ or/and $v$ become blue or red, $w$ will not be contained in any special triangles anymore and $f(w) \le 9$ will be valid. 
Further, as the case (iv) is excluded,  if $u,v \in W_1$ in $G_{i-1}$, then $w\in B_2$. If $u$  becomes dominated, $w$ will be either red or a blue vertex of degree 1. Thus, $f(w)$ decreases by at least 2.

\paragraph{Claim (b)} The weight of the graph decreases by at least 20 in every turn.\\
This statement clearly holds if a white vertex is played. Assume that a vertex $v \in B$ is selected. If $\deg_{G_{i-1}}^W(v)$ equals 2 or 3, the decrease in $\w(G_{i-1})$ would be at least $9+2\cdot 11=31$.  Also, if $v$ has a white neighbor $u$ such that  $\deg_{G_{i-1}}^W(u)=0$, the decrease is $g_i \ge 8+20=28$. Hence, we may assume that there is an induced path $vuw$ such that $v\in B_1$ and $u, w \in W$. Then, the edge $uw$ belongs to a special triangle $uww'$. When $v$ is played, $v$ is recolored red and $f(v)$ decreases by 8; $u$ is recolored blue and $f(u)$ decreases by at least 11 as it cannot be contained in any special triangles in $G_{i}$. Moreover, $w'$ also becomes a  vertex which is not contained in any special triangles, therefore $f(w')$ decreases by at least  1. These imply $g_i \ge 8+11+1=20$.

\paragraph{(vi)} There exists an edge $vu$ in $G_{i-1}$ such that  $v \in W_0$, $u\in B_3$, and none of (i)--(v) are valid.\\
Let $w_1$ and $w_2$ be the further white neighbors of $u$ and assume that Dominator plays $u$. Clearly, if $w_i \in W_0$ (for $i=1$ or 2), at least three vertices, namely $u$, $v$ and $w_i$ are recolored red and the decrease in the weights is $g_i \ge 9+2\cdot 20 +11 >52$. Now, assume that $w_1$ and $w_2$ are adjacent, i.e.\ $w_1w_2u$ is a special triangle. Then, $f(u)=10$ and $w_1, w_2 \in  W_2$ in $G_{i-1}$. After $u$ is selected, $f(u)$ decreses by $10$, $f(v)$ decreases by 20 and both $w_1$ and $w_2$ become a blue vertex of W-degree 1. Hence, $g_i \ge 10+20+2\cdot 11 =52$.
If $w_1$ and $w_2$ are not adjacent and $p_i=u$, then $f(u)$ decreases by 9 and $f(v)$ decreases by 20. First assume that $w_i \in W_1$ (for $i=1$ or $2$). After playing $u$, the vertex $w_i$ will be a blue vertex of W-degree 1. Hence, $f(w_i)$ decreases by 12. In the other case, $w_i \in W_2$ in $G_{i-1}$. Let $x_{i,1}$ and $x_{i,2}$ be the two white neighbors of $w_i$. By Claim (a), there exist blue vertices $y_{i,1}$ and $y_{i,2}$ ($y_{i,1}\neq y_{i,2}$) such that $f(y_{i,1})$ and $f(y_{i,2})$ decreases by at least 1 when $w_i$ becomes blue. Together with the decrease of 11 in $f(w_i)$, it also yields a decrease of $12$ associated to the vertex $w_i$. Hence, when $u$ is not contained in a special triangle,  $g_i \ge 9+20+2\cdot 12=53$. By Claim (b), $g_i \ge 20$ and hence $g_i+g_{i+1}\ge 72$ holds for all cases.

\paragraph{(vii)} There exists a vertex $v \in W_2$ in $G_{i-1}$,  and none of (i)--(vi) hold.\\
The vertex $v\in W_2$ must be contained in a 2-regular component of $G_{i-1}[W]$ which is a cycle $v_1,\dots,v_j$ of length at least 4. Assume that $v=v_1$ and Dominator plays this vertex. Since $v_1$ is recolored red and both $v_j$, $v_2$ belong to $B_1$ in $G_i$, $f(v_1)$, $f(v_j)$
and $f(v_2)$ decrease by 20, 12, 12, respectively. Moreover, consider the blue vertices $w_j$, $w_1$, $w_2$, $w_3$ forming special triangles with the vertex pairs $(v_{j-1},v_j)$, $(v_{j},v_1)$, $(v_{1},v_2)$, and $(v_{2},v_3)$, respectively. By Claim~(a), $f(w_j)$ and $f(w_3)$  decrease by at least 1 each. For $i=1,2$, the W-degree of $w_i$ decreases by 2; and hence, $w_i \in R\cup B_1$ in $G_i$. Thus, $f(w_i)$ decreases by at least 2 for $i=1, 2$. Further, we observe that at least one of $w_1$ and $w_2$ had to be blue already in $G_{i^*}$. Indeed, otherwise the white vertex $v_1$ would be of W-degree 4 in $G_{i^*}$ that contradicts the definition of Phase~1. Hence, $f^+(w_1)+f^+(w_2)\ge 2$ in $G_{i-1}$ and, as they both are contained in $R\cup B_1$ in $G_i$, $f^+(w_1)+f^+(w_2)=0$ in $G_i$. We infer that $g_i \ge 20+2\cdot 12 +2\cdot 1 +2\cdot 2+2=52$ and we have $g_i+g_{i+1}\ge 52+20=72$.

\medskip

Consequently, if none of (i)--(vii) hold, $W=W_0 \cup W_1$ in the residual graph $G_{i-1}$.
Under the same condition, we prove next that $g_j \ge 22$ holds for every $j \ge i$.
Assume that a vertex $p_j$ is played in $G_{j-1}$. If $p_j\in W_1$, the vertex $p_j$ and its white neighbor are recolored red and $g_j \ge 40$. If $p_j \in W_0$, it has at least two blue neighbors, since it was contained in a triangle in $G_0$. Further, as case (vi) is excluded, these blue neighbors, namely $u_1$ and $u_2$, are from $B_1\cup B_2$. Hence, when $p_j$ is recolored red, each of  $f(u_1)$ and $f(u_2)$  changes either from 9 to 8, or from 8 to 0. Hence, $g_j \ge 20+2\cdot 1=22$. If $p_j \in B_2 \cup B_3$, then $p_j$ is recolored red and its neighbors are recolored blue or red. This gives $p_j \ge 9+ 2\cdot 12=33$. If $p_j \in B_1$ and its neighbor is from $W_0$, both vertices turn to red and $g_j \ge 8+20$. The only remaining case is when $p_j\in B_1$ and its neighbor $w$ belongs to $W_1$. Denote by $v$ the white neighbor of $w$. By Claim (a), there exists a special triangle incident with the edge $wv$, we denote its blue vertex  by $u$. Since $w$ and $v$ belong to $W_1$ in $G_{j-1}$, by Claim (a),  $f(u)$ decreases by at least 2 when $w$ is recolored blue. In this case, again, we have $g_j \ge 8+12+2=22$ that proves the statement.

\paragraph{(viii)} There exists a vertex $v \in W_1$ in $G_{i-1}$,  and none of (i)--(vii) hold.\\ 
In this case we have a pair $v,u$ of adjacent white vertices. Let $w$ be the third vertex of a triangle which contains the edge $uv$ in $G^*$. As case (vi) is excluded, $w$ is  contained in $ B_2\cap B_T$ in $G_{i-1}$. If Dominator plays $w$, all the three vertices $u$, $v$ and $w$ become  red and $g_i \ge 10+ 2\cdot 20=50$. This implies $g_i+g_{i+1}\ge 50+22 =72$ as stated. 

\medskip
If $(i)$--$(viii)$ do not hold, all the white vertices are of W-degree zero, and by $(vi)$ we may not have blue vertices of degree higher than 2.

\paragraph{(ix)} There exists a vertex $v \in B_2$ in $G_{i-1}$,  and none of (i)--(viii) hold.\\ 
Such a vertex $v$ has two white neighbors, say $u$ and $w$, from $W_0$. Moreover, $u$ and $w$ must have blue neighbors $u', w'$ which are different from $v$. If Dominator plays $v$, all the three vertices $v,u, w$ are recolored red.  Further, the W-degrees of $u'$ and $w'$ are decreased. If $u'=w'$ then $f(u')$ decreases by 9; if $u'\neq v'$, $f(u')$ and $f(w')$ decreases by at least 1 each.  Hence, we have $g_i \ge 9+2\cdot 20 +2= 51$. This implies $g_i+g_{i+1} \ge 51+22=73$.

\paragraph{(x)} None of (i)--(ix) hold.\\ 
Since each white vertex belongs to $W_0$ and each blue vertex is in $B_1$, the residual graph $G_{i-1}$ consists of star components, each of them having white centers and blue leaves. If $v$ is white in $G_{i-1}$, it has no red neighbors and $\deg^{WB}_{G_{i-1}}(u)=\deg_{G^*}(u) \ge 2$. Under the present conditions, with each move (taken by either Dominator or Staller) exactly one component of $G_{i-1}$ becomes completely red. Thus, $g_i\ge 20+2\cdot 8 \ge 36$, and if the game is not finished, also $g_{i+1} \ge 36$ holds.

 We have shown that $g_i +g_{i+1}\ge 72$ and also $g_i \ge 36$ are true in all the ten cases. Since these cases (i)--(x) together cover all possibilities,  Lemma~\ref{lem-Ph2} is proved. \qed
 
 \medskip
 
By  Lemma~\ref{lem-Ph2}, the average decrease of the weight of the residual graph in a turn is at least $36$, if Dominator plays greedily. Since $\w(G_0)=20n$ at the beginning of the game and $\w(G_k)=0$ at the end, Dominator can make sure that $k \le 20n/36$ holds for the length $k$ of the game. Therefore, 
$$ \gamma_g(G^*)\le k  \le \frac{20n}{36}= \frac{5n}{9}$$
that proves Theorem~\ref{5/9-gen} (ii).
By Proposition~\ref{equiv1} (ii), this is equivalent with the statement (i), that finishes the proof of Theorem~\ref{5/9-gen}. \qed 

\subsection*{The Staller-start game}

If Staller starts the game, we can use the same weight assignment as in the proof of Theorem~\ref{5/9} and follow the flow of Lemmas \ref{lem-Ph1} and \ref{lem-Ph2}. The only difference is that we insert a preliminary turn belonging to Phase 1, which contains the first move of Staller. Here, whichever vertex $p_0$ is played by Staller, $p_0$ is recolored red and at least two neighbors are recolored blue. Hence, the corresponding decrease in the weight of $G^*$ is $g_0 \ge 20+2\cdot 8=36$. This implies the following theorem analogous to Theorem~\ref{5/9}.

\begin{theorem} \label{5/9-Staller}
\begin{itemize}
\item[$(i)$]
If $\cH$ is an isolate-free 3-uniform hypergraph on $n$ vertices, then
$$\gamma_g'(\cH)\le \frac{5}{9}n.$$
\item[$(ii)$] 
If $G$ is an isolate-free graph of order
$n$ and each edge of $G$ belongs to a triangle, then
$$\gamma_g'(G)\le \frac{5}{9}n.$$
\end{itemize}
\end{theorem}
\medskip

\section{Conclusions and open problems}

If $\cH$ is a hypergraph and each edge of $\cH$ contains at least three vertices, then in the 2-section graph $[\cH]_2$  every edge is contained in a triangle. Therefore, our Theorem~\ref{5/9} can also be stated in the following more general form.

\begin{theorem} \label{5/9-gen}
If $\cH$ is a hypergraph on $n$ vertices which contains neither isolated vertices nor edges of size smaller than 3, then
$\gamma_g(\cH)\le \frac{5}{9}n.$
\end{theorem}

In particular, also for every isolate-free $k$-uniform hypergraph with $k \ge 4$, the $5n/9$ upper bound is valid. For $k\ge 5$ we can give better estimates by earlier results. Indeed, since  the 2-section graph of a 5-uniform hypergraph has minimum degree $\delta\ge 4$ and for such graphs $\gamma_g \le 37n/72 < 0.5139 n$ was proved in \cite{B2}, we have the same upper bound on the game domination number of 5-uniform hypergraphs. Similarly, by the results of \cite{B2}, for 6-uniform hypergraphs $\gamma_g \le 2102n/4377 < 0.4803 n$ holds; and for 7-uniform hypergraphs  $\gamma_g < 0.4575 n$. On the other hand, we do not have evidence that any of these upper bounds, including the bound $5n/9$ in Theorem~\ref{5/9}, is tight. 

\begin{problem} Improve the above upper bounds for the game domination numbers of $k$-uniform hypergraphs with $k \ge 3$.
\end{problem}

The best lower bounds that we have for small values of $k$ come from the following constructions:

\vspace{2mm}

\noindent
\textbf{Construction 1} (generalization of the construction given in the proof of Proposition \ref{3unilower}): Let $\cH_{k,1}$ be the hypergraph on $k^2$ vertices with vertex set $V(\cH_{k,1})=\{(i,j):1\le i,j\le k\}$ and edge set $E(\cH_{k,1})=\{\{(i,j):j=j_0\}:1\le j_0\le k\}\cup \{\{(i,j):i=i_0\}:1\le i_0\le k-1\}$ (see Figure \ref{fig:hk1}). In a domination game which is played on $\cH_{k,1}$, Staller may use the following strategy: after any move $(i,j)$ of Dominator he plays an $(i,j')$ if there exists such a legal move. Hence, $\gamma_g(\cH_{k,1})\ge k+ \lfloor\frac{k-1}{2}\rfloor$. On the other hand, while it is possible, Dominator may insist on playing vertices of degree two  no neighbors of which have been played so far in the game. This proves $\gamma_g(\cH_{k,1})= k+ \lfloor\frac{k-1}{2}\rfloor$, and it can be proved that $\gamma_g'(\cH_{k,1})=\gamma_g(\cH_{k,1})$. Note that in the Staller-start version, the optimal first move is a vertex of degree one.

\begin{figure}[h]
 \begin{center}
 \advance\leftskip-4mm
 \includegraphics[width=8.2cm]{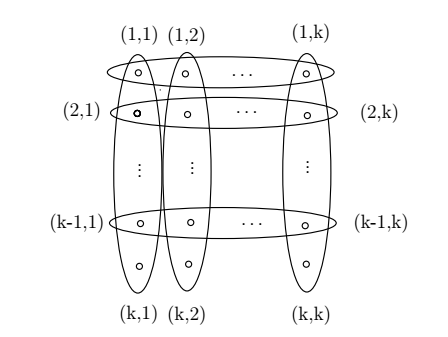}
  \caption{$\cH_{k,1}$}
  \label{fig:hk1}
  \end{center}
\end{figure}

\noindent
\textbf{Construction 2}: Let $\cH_{k,2}$ be the hypergraph on $2k+1$ vertices with vertex set $V(\cH_{k,2}):=\{a,b,c,u_1,u_2,\dots,u_{k-1},v_1,v_2,\dots, v_{k-1}\}$ and edge set $E(\cH_{k,2}):=\{\{a,u_1,u_2,\dots,u_{k-1}\},$  

\noindent
$\{b,u_1,u_2,\dots,u_{k-1}\}, \{b,v_1,v_2,\dots, v_{k-1}\}, \{c,v_1,v_2,\dots, v_{k-1}\}\}.$ It is easy to check that 

\noindent
$\gamma_g(\cH_{k,2})=\gamma_g'(\cH_{k,2})=3$.

For a $k$-uniform hypergraph $\cF$ with vertex set $x_1,\dots,x_t$ take $t$ disjoint copies of $\cH_{k,2}$, namely $\cH_{k,2}^1,\dots ,\cH_{k,2}^t$, and identify $x_i$ with the vertex $b^i$ from $H_{k,2}^i$ for $1\le i \le t$. This way, the $k$-uniform hypergraph $\cF(\cH_{k,2})$ is obtained on $n(2k+1)$ vertices (see Figure \ref{fig:hk2}). In the domination game, Staller can achieve that each $b^i$ is played. Then, at least three vertices must be played from each $\cH_{k,2}^i$ during the game. On the other hand, any reasonable strategy of Dominator can ensure that the game is not longer than  $3t$. Hence, $\gamma_g(\cF(\cH_{k,2}))=\gamma_g'(\cF(\cH_{k,2}))=3t$.

\begin{figure}[H]
 \begin{center}
 \advance\leftskip-4mm
 \includegraphics[width=8.2cm]{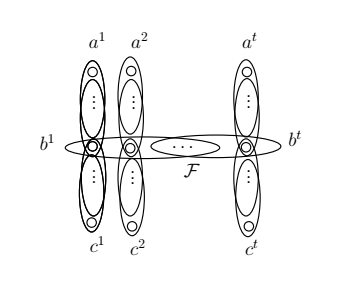}
  \caption{$\cH_{k,2}$}
  \label{fig:hk2}
  \end{center}
\end{figure}

\smallskip

Note that for $k$ odd Construction 1 gives better lower bound while for $k$ even Construction 2 is the stonger one.

\medskip

\section*{Acknowledgement}

This research was supported by the National Research, Development and Innovation Office -- NKFIH, grant SNN 116095.
Research of B.\ Patk\'os was partly supported by the J\'anos Bolyai Research Fellowship of the Hungarian Academy of Sciences.

\end{document}